\documentclass[10pt,reqno]{amsart}
\usepackage{latexsym}
\usepackage{graphicx}
\usepackage{fancyhdr,amssymb}

\newcommand{\R}{{\mathbb R}}

\newcommand{\C}{{\mathbb C}}

\theoremstyle{plain}
\numberwithin{equation}{section}
\newtheorem{thm}{Theorem}[section]
\newtheorem{theorem}[thm]{Theorem}
\newtheorem{lemma}[thm]{Lemma}
\newtheorem{definition}[thm]{Definition}
\newtheorem{proposition}[thm]{Proposition}

\begin{document}

\setcounter{page}{1}

\title[OMEGA Integral]{Direct Proofs of the Fundamental Theorem of Calculus for the Omega Integral}
\author[C.~B.~Dawson]{C. Bryan Dawson}
\address{Department of Mathematics\\
                Union University\\
                1050 Union University Dr.\\
                Jackson, TN\\
                38305}
\email{bdawson@uu.edu}
\author[M.~Dawson]{Matthew Dawson}
\address{CONACYT---CIMAT Unidad M\'{e}rida\\
Centro de Investigaci\'{o}n en Matem\'{a}ticas, A.C.\\
UNIDAD M\'{E}RIDA\\
Parque Cient\'{i}fico y Tecnol\'{o}gico de Yucat\'{a}n\\
Km 5.5  Carretera Sierra Papacal --- Chuburn\'{a} Puerto\\
Sierra Papacal; M\'{e}rida, Yucat\'{a}n. CP 97302}
\email{matthew.dawson@cimat.mx}

\begin{abstract}
When introduced in a 2018 article in the American Mathematical Monthly, the omega integral was shown to be an extension of the Riemann integral.  Although results for continuous functions such as the Fundamental Theorem of Calculus follow immediately, a much more satisfying approach would be to provide direct proofs not relying on the Riemann integral.  This note provides those proofs.
\end{abstract}

\maketitle

\newcommand{\st}{\mathop{\rm st}}

\section{Introduction}

The omega integral, which makes use of the hyperreals to integrate real functions, was introduced in \cite{monthly} and proven to be an extension of the Riemann integral.  Theorems relating to continuous functions, such as integrability, additivity, and the fundamental theorem follow immediately. 
It is natural to seek direct proofs of such theorems (that is, proofs that do not rely on the Riemann integral).
The purpose of this note is to provide these proofs.

We will make use of the following notational conventions from \cite{monthly}: 0 will not be considered to be an infinitesimal and $\Omega$ is often used to represent a positive unlimited hyperreal integer and $\omega:=1/\Omega$ for its infinitesimal reciprocal.  In addition, capital Greek letters generally represent unlimited hyperreals, and lowercase Greek letters generally represent infinitesimals.  
Otherwise, notation and basic facts about the hyperreals will follow \cite{goldblatt}, although we write the standard part operator $\st(\cdot)$ in place of the shadow. 
All references to integrability in this article will refer to the omega integral.  For convenience we restate below the definition of the omega integral from \cite{monthly}, which is based on right-hand sums for equal-width partitions using an unlimited number of subintervals.

First, we remind the reader that any function $f:[a,b]\rightarrow \R$ defined on a subinterval of $\R$ admits a natural extension to a function ${}^*f:{}^*[a,b]\rightarrow {}^*\R$ defined on the hyperreal interval ${}^*[a,b] = \{x\in {}^*\R \mid a\leq x\leq b\}$ (see \cite{goldblatt}).  For convenience, we will refer to both the original function $f$ and its hyperreal extension ${}^*f$ by the same name $f$.  The omega integral is then defined as follows:

\begin{definition}[Omega integral] Let $f:[a,b]\rightarrow\R$ be a
function. If there exists $L\in \R
$ such that $\st\left(\sum_{k=1}^\Omega f(x_k)\Delta x\right)=L$ for every positive unlimited hyperreal integer $\Omega$ (where $\Delta x=\frac{b-a}{\Omega}$ and $x_k=a+k\Delta x$ for $k=1,2,\dots,\Omega$), then we write
$$\int_a^b f(x)\,dx =L$$
and call $f$ integrable.  If $\sum_{k=1}^\Omega f(x_k)\Delta x$ is positive infinite for every $\Omega$ we write $\int_a^b f(x)\,dx =\infty$, and for negative infinite sums we use $-\infty$. 
\end{definition}



We call $\sum_{k=1}^\Omega f(x_k)\Delta x$ an \emph{omega sum} irregardless of the name of the unlimited hyperreal integer used therein.

The following lemma will also be helpful.

\begin{lemma}\label{L1} Let $\Omega$ be a positive unlimited integer, let $M$ be a positive real number, and let $\gamma_k$ be 0 or infinitesimal for $k=1,\dots,\Omega$.  Then $\sum_{k=1}^\Omega \gamma_k M\frac{1}{\Omega}$ is 0 or infinitesimal.\end{lemma}

\begin{proof}
Let $\varepsilon>0$ be real. Since $\gamma_k$ is 0 or infinitesimal, $-\frac{\varepsilon}{M}<\gamma_k<\frac{\varepsilon}{M}$.  Hence
  $\sum_{k=1}^{\Omega}-\frac{\varepsilon}{M}<
\sum_{k=1}^{\Omega}\gamma_k<\sum_{k=1}^{\Omega}\frac{\varepsilon}{M}$, and multiplying by the positive constant $M\frac{1}{\Omega}$ gives
 $\sum_{k=1}^{\Omega}\frac{-\varepsilon}{\Omega}<   \sum_{k=1}^{\Omega}\gamma_k M\frac{1}{\Omega}<\sum_{k=1}^{\Omega}\frac{\varepsilon}{\Omega}=\frac{\varepsilon}{\Omega}\Omega=\varepsilon$.  Since $\varepsilon$ was arbitrary, $\sum_{k=1}^\Omega \gamma_k M\frac{1}{\Omega}$ is 0 or infinitesimal.
\end{proof}

\section{Integrability of Continuous Functions}

Intuitively, a partition using an unlimited number of subintervals only leaves room for an approximation to a continuous function to stray infinitesimally far from the value of the function on real numbers, and on a finite-length interval the resulting approximation is within an infinitesimal of the value we seek.  To prove that any two such sums have the same standard part, we show that the sum for a common refinement of the partitions is within an infinitesimal of each sum.

\begin{theorem}\label{A} Let $f$ be continuous on $[a,b]$.  Then $f$ is integrable.
\end{theorem}

\begin{proof}  Let $f$ be continuous on $[a,b]$.  Then $f$ is bounded and each omega sum is limited.

Let $\Omega$ be a positive unlimited hyperreal integer and consider its omega sum: 
\begin{equation}\label{1}
\sum_{k=1}^\Omega f(x_k)\Delta x, \text{ where } \Delta x=(b-a)\cdot\frac{1}{\Omega} \text{ and }x_k=a+k(b-a)\cdot\frac{1}{\Omega}.
\end{equation}
Then for any positive unlimited hyperreal integer $B$, we form the omega sum based upon $B\Omega$ subintervals, calling the partition points $y_k$ to distinguish them from the partition points of the previous sum:
\begin{equation}\label{2}
\sum_{k=1}^{B\Omega}f(y_k)\Delta y,\text{ where } \Delta y=(b-a)\cdot\frac{1}{B\Omega} \text{ and } y_k=a+k(b-a)\cdot\frac{1}{B\Omega}.\end{equation}

We wish to show that the omega sums in \eqref{1} and \eqref{2} differ by at most an infinitesimal.  To that end, fix $n$ with $0\le n\le\Omega-1$.  If $1\le p\le B$, then 
\begin{align*}
x_n=y_{Bn}<y_{Bn+p} &=a+(Bn+p)(b-a)\frac{1}{B\Omega}\\
&=a+n(b-a)\frac{1}{\Omega}+p(b-a)\frac{1}{B\Omega}\\
&=x_n+(b-a)\frac{p}{B}\frac{1}{\Omega}\\
&\le x_n+(b-a)\frac{1}{\Omega}=x_{n+1}.\end{align*}
Since $x_n$ and $x_{n+1}$ differ by an infinitesimal and $f$ is continuous, 
$$f(y_{Bn+p})=f(x_{n+1})+\gamma_{n,p}$$
for each such $n$ and $p$, where $\gamma_{n,p}$ is 0 or infinitesimal (see \cite{goldblatt}).  Then
\begin{align*}
\sum_{k=1}^{B\Omega}f(y_k)\Delta y&=\sum_{n=0}^{\Omega-1}\sum_{p=1}^B f(y_{Bn+p})(b-a)\frac{1}{B\Omega}\\
&=\sum_{n=0}^{\Omega-1}\sum_{p=1}^B\left(f(x_{n+1})+\gamma_{n,p}\right)(b-a)\frac{1}{B\Omega}\\
&=\sum_{n=0}^{\Omega-1}\sum_{p=1}^Bf(x_{n+1})(b-a)\frac{1}{B\Omega}+\sum_{n=0}^{\Omega-1}\sum_{p=1}^B\gamma_{n,p}(b-a)\frac{1}{B\Omega}\\
&=\sum_{n=0}^{\Omega-1}B\cdot f(x_{n+1})(b-a)\frac{1}{B\Omega}+\sum_{n=0}^{\Omega-1}\sum_{p=1}^B\gamma_{n,p}(b-a)\frac{1}{B\Omega}\\
&=\left(\sum_{n=1}^\Omega f(x_n)(b-a)\frac{1}{\Omega}\right)+\sum_{n=0}^{\Omega-1}\sum_{p=1}^B\gamma_{n,p}(b-a)\frac{1}{B\Omega}\\
&=\sum_{k=1}^\Omega f(x_n)\Delta x+\gamma
\end{align*}
where $\gamma$ is 0 or infinitesimal, the last step by lemma \ref{L1}.

Now consider any two omega sums 
\begin{equation}\label{7}
\sum_{k=1}^{B_1}f(x_k)\Delta x\end{equation}
 and \begin{equation}\label{8}\sum_{k=1}^{B_2}f(y_k)\Delta y,\end{equation}
  where $B_1$ and $B_2$ are arbitrary positive unlimited integers.  By the previous work, the omega sum based on the common partition refinement for $B_1B_2$,
$$\sum_{k=1}^{B_1B_2}f(z_k)\Delta z,$$
 has the same standard part as each of the omega sums in \eqref{7} and \eqref{8}.  Consequently all omega sums have the same standard part and $f$ is integrable.\end{proof}

\section{Additivity for Continuous Functions}

In \cite{monthly} it is shown that the omega integral is not additive in general, 
but through the Riemann integral we know that the omega integral is additive for continuous functions.
The difficulty presented in a direct proof using omega sums is that the partition points of an omega sum for the integral on the right in Theorem \ref{B} might not match the partition points for any omega sum for one or both integrals on the left, a situation that occurs whenever $\frac{c-a}{b-a}$ is irrational.  
In a similar manner to the proof of Theorem \ref{A}, as long as our partition points are infinitesimally close then the corresponding values of $f$ will be as well, allowing us to show the appropriate omega sums are close enough.

\begin{theorem}\label{B}
 If $f$ is continuous on $[a,c]$ and $a<b<c$ then
$$\int_a^b f(x)\,dx+\int_b^c f(x)\,dx=\int_a^c f(x)\,dx.$$
\end{theorem}

\begin{proof} Let $\Omega$ be a positive unlimited integer and form the associated omega sum for $f$ on $[a,c]$:
$$\sum_{k=1}^\Omega f(x_k)\Delta x, \text{ where } \Delta x=(c-a)\cdot\frac{1}{\Omega} \text{ and }x_k=a+k(c-a)\cdot\frac{1}{\Omega}.$$
Since $a<b<c$, there exists a positive unlimited integer $B<\Omega$ such that
\begin{equation}\label{12}
x_B=a+B\Delta x\le b<a+(B+1)\Delta x=x_{B+1}.\end{equation}
Using $B$ subintervals we may form an omega sum for $f$ on $[a,b]$:
$$\sum_{k=1}^B f(y_k)\Delta y, \text{ where } \Delta y=(b-a)\cdot\frac{1}{B} \text{ and }y_k=a+k(b-a)\cdot\frac{1}{B}.$$
Then although the partition points $x_k$ and $y_k$ may not be equal, they differ by at most an infinitesimal; in fact, for each $k\in\{1,2,\dots,B\}$,
$$\left|x_k-y_k\right|\le\left|x_B-y_B\right|=\left|x_B-b\right|<\Delta x,$$
the last inequality stemming from \eqref{12}.
Then $|x_k-y_k|$ is at most infinitesimal, and since $f$ is continuous
$|f(x_k)-f(y_k)|$ is also at most infinitesimal.  Therefore
\begin{equation}\label{15}
\begin{split}
&\left|\sum_{k=1}^B f(y_k)\Delta y-\sum_{k=1}^B f(x_k)\Delta x\right|\\
&\qquad=\left|\sum_{k=1}^B f(y_k)\Delta y-\left(\sum_{k=1}^B f(x_k)\Delta y+\sum_{k=1}^Bf(x_k)(\Delta x-\Delta y)\right)\right|\\
\qquad&\le\left|\sum_{k=1}^B\left(f(y_k)-f(x_k)\right)\Delta y\right|+\left|\sum_{k=1}^B f(x_k)(\Delta x-\Delta y)\right|.\end{split}\end{equation}
The term on the left is infinitesimal by Lemma \ref{L1}. 
To see how $\Delta x$ and $\Delta  y$ compare, use the fact that $b=y_B=a+B\Delta y$ combined with \eqref{12},
$$a+B\Delta x\le a+B\Delta y<a+(B+1)\Delta x,$$
to conclude that 
$$0\le\Delta y-\Delta x<\frac{\Delta x}{B}.$$
Then the remaining term from \eqref{15} yields
\begin{align*}
\left|\sum_{k=1}^B f(x_k)(\Delta x-\Delta y)\right|
&=\left|\Delta x-\Delta y\right|\cdot\left|\sum_{k=1}^B f(x_k)\right|\\
&<\frac{\Delta x}{B}\left|\sum_{k=1}^B f(x_k)\right|\\
&=\frac 1B\left|\sum_{k=1}^B f(x_k)\Delta x\right|\\
&\le \frac1B\sum_{k=1}^B|f(x_k)|\Delta x\\
&\le\frac1B\sum_{k=1}^\Omega|f(x_k)|\Delta x\\
&=\frac 1B\left(\int_a^c |f(x)|\,dx+\text{infinitesimal}\right),
\end{align*}
since $|f|$ is continuous and hence integrable by Theorem \ref{A}.  Therefore the last quantity in the above inequalities is also at most infinitesimal, allowing the conclusion that
\begin{equation}\label{20}
\left|\sum_{k=1}^B f(y_k)\Delta y-\sum_{k=1}^B f(x_k)\Delta x\right|
\end{equation}
is infinitesimal.

Now form the omega sum for $f$ on $[b,c]$ using $\Omega-B$ subintervals:
$$\sum_{k=1}^{\Omega-B} f(z_k)\Delta z, \text{ where } \Delta z=(c-b)\cdot\frac{1}{\Omega-B} \text{ and }z_k=b+k(c-b)\cdot\frac{1}{\Omega-B}.$$
A similar argument to the one above shows that
\begin{equation}\label{22}
\left|\sum_{k=1}^{\Omega-B}f(z_k)\Delta z-\sum_{k=B+1}^{\Omega}f(x_k)\Delta x\right|\end{equation}
is infinitesimal.
Combining \eqref{20} and \eqref{22},
\begin{align*}
\sum_{k=1}^\Omega f(x_k)\Delta x&=\sum_{k=1}^Bf(x_k)\Delta x+\sum_{k=B+1}^\Omega f(x_k)\Delta x\\
&=\sum_{k=1}^Bf(y_k)\Delta y+\gamma_1+\sum_{k=1}^{\Omega-B}f(z_k)\Delta z+\gamma_2\end{align*}
where $\gamma_1$ and $\gamma_2$ are 0 or infinitesimal.
But since $f$ is continuous, by Theorem \ref{A} each of these omega sums is within an infinitesimal of its associated integral and the definition of omega integral yields
$$\int_a^c f(x)\,dx=\int_a^b f(x)\,dx+\int_b^c f(x)\,dx.$$
\end{proof}

\section{Fundamental Theorem of Calculus, Part I}

The usual proof of the fundamental theorem given in a calculus course still works for the omega integral.  There are, however, a few details to be careful about.

\begin{theorem}\label{C} If $f$ is continuous on $[a,b]$, then the function
$$F(x)=\int_a^x f(t)\,dt$$
is differentiable on $[a,b]$ and $F'(x)=f(x)$.  
\end{theorem}

\begin{proof}
Let $f$ be continuous on $[a,b]$.  Then for any $c\in[a,b]$, by Theorem \ref{A} we may define
$F_c:[a,b]\to{\bold R}$ by
$$F_c(x)=\int_c^x f(t)\,dt.$$
As with any real function, $F_c$ extends to the domain ${}^*[a,b]$ consisting of hyperreals between $a$ and $b$ inclusive.  We can then associate the integral notation with the extended function; that is, for any hyperreal $x\in{}^*[a,b]$,
$$\int_c^x f(t)\,dt=F_c(x).$$
The transfer property allows us to extend the properties of integrals as well.
For instance, for each real $c\in[a,b]$, we write
$$\forall x\in[a,b] \left(m\le f(t)\le M\,\, \forall t\in[c,x]\right)\Longrightarrow m(x-c)\le F_c(x)\le M(x-c)$$
transfers to
\begin{equation}\label{34}
\forall x\in{}^*[a,b] \left(m\le f(t)\le M\,\, \forall t\in{}^*[c,x]\right)\Longrightarrow m(x-c)\le F_c(x)\le M(x-c).\end{equation}
The additivity we will also need transfers.  Fixing $x\in[a,b]$ for each real $\alpha$ such that $x+\alpha\in[a,b]$, we have
\begin{equation}\label{35}
\int_a^x f(t)\,dt+\int_x^{x+\alpha}f(t)\,dt=\int_a^{x+\alpha}f(t)\,dt.\end{equation}
By transfer the same is true for any hyperreal $\alpha$ such that $x+\alpha\in{}^*[a,b]$.

Fix $x\in[a,b]$.
In order to prove that $F'(x)=f(x)$ we need to show that for any infinitesimal $\alpha$,
$$\st\left(\frac{F(x+\alpha)-F(x)}{\alpha}\right)=f(x).$$
To that end, let $\alpha$ be an infinitesimal such that $x+\alpha\in {}^*[a,b]$.
Using the definition of $F$ and \eqref{35},
\begin{equation}\label{37}
\frac{F(x+\alpha)-F(x)}{\alpha}=\frac{\int_a^{x+\alpha}f(t)\,dt-\int_a^x f(t)\,dt}{\alpha}=\frac{\int_x^{x+\alpha}f(t)\,dt}{\alpha}.\end{equation}
Since $f$ is continuous, 
 $|f(t)-f(x)|$ is infinitesimal for each $t$ between $x$ and $x+\alpha$.  Thus for each positive real $\varepsilon$,
$$f(x)-\varepsilon<f(t)<f(x)+\varepsilon$$
for each such $t$, and by \eqref{34}
$$\left(f(x)-\varepsilon\right)\alpha\le\int_x^{x+\alpha}f(t)\,dt\le\left(f(x)+\varepsilon\right)\alpha.$$
Dividing by $\alpha$ and combining with \eqref{37} shows that
$$f(x)-\varepsilon<\frac{F(x+\alpha)-F(x)}{\alpha}<f(x)+\varepsilon,$$
or
$$\left|\frac{F(x+\alpha)-F(x)}{\alpha}-f(x)\right|<\varepsilon.$$
As $\varepsilon$ is arbitrary, $\left|\frac{F(x+\alpha)-F(x)}{\alpha}-f(x)\right|$ is infinitesimal for each $\alpha$ and we conclude that $F'(x)=f(x)$.
\end{proof}

Notice that throughout the proof the lower limits of integration are always real, so that there is no need to extend to functions $F_c$ for $c\in{}^*[a,b]$.

\section{Fundamental Theorem of Calculus, Part II}

The usual calculus-level proof of FTC~II from FTC~I applies.  However, just as with the Riemann integral, it is also possible to give a direct proof of FTC~II, one that does not require proving FTC~I first.  We 
begin with a lemma.

\begin{lemma}\label{L2}
If $H'=f$ is continuous on $[a,b]$, then for any hyperreal $x\in{}^*[a,b]$ and infinitesimal $\alpha$ for which $x+\alpha\in{}^*[a,b]$, one has
$$H'(x)=\frac{H(x+\alpha)-H(x)}{\alpha}+\gamma$$
 where $\gamma$ is 0 or infinitesimal.
\end{lemma}

\begin{proof}
Let $x,x_0\in{\mathbf R}$ such that both $x$ and $x+x_0$ are in $[a,b]$. Then the mean value theorem may be applied to $H$ since $H$ is differentiable on $[a,b]$.  Thus there is some $c\in{\mathbf R}$ such that $c$ is between $x$ and $x+x_0$ and
 $x_0H'(c)=H(x+x_0)-H(x)$.
Then by transfer,
for all $x,x_0\in{}^*{\mathbf R}$ such that $x,x+x_0\in {}^*[a,b]$, there exists $c\in{}^*{\mathbf R}$ such that $c$ is between $x$ and $x+x_0$ and  $x_0H'(c)=H(x+x_0)-H(x)$.

Let $x$ be a hyperreal in ${}^*[a,b]$ and let $\alpha$ be an infinitesimal such that $x+\alpha\in{}^*[a,b]$.  Then there is some hyperreal $c$ between $x$ and $x+\alpha$ such that $\alpha H'(c)=H(x+\alpha)-H(x)$.  But since $H'=f$ is continuous on $[a,b]$  and $c$ differs from $x$ by an infinitesimal, $H'(c)=H'(x)+\gamma$ where $\gamma$ is 0 or infinitesimal. 
Hence
$H(x+\alpha)-H(x)=\alpha H'(c)=\alpha(H'(x)+H'(c)-H'(x))=\alpha H'(x)+\alpha\gamma$ and the conclusion follows.

\end{proof}

\begin{theorem}\label{D} If $f$ is continuous on $[a,b]$ and $H$ is any antiderivative of $f$ on $[a,b]$, then
$$\int_a^b f(x)\,dx=H(b)-H(a).$$
\end{theorem}

\begin{proof} We are given that $H'(x)=f(x)$ on $[a,b]$.  Given a positive unlimited integer $\Omega$, its associated omega sum is 
\begin{equation}\label{51}\sum_{k=1}^\Omega H'(x_k)\Delta x=
\sum_{k=1}^\Omega H'(a+k(b-a)\omega)\cdot (b-a)\omega,\end{equation}
writing $\omega$ for $\frac{1}{\Omega}$.

Then using Lemma \ref{L2} with 
 $x=a+k(b-a)\omega$ and $\alpha=(b-a)\omega$,
\begin{eqnarray*} &  & \sum_{k=1}^\Omega H'(a+k(b-a)\omega)\cdot (b-a)\omega\\
& = & \left(\sum_{k=1}^{\Omega-1}\left(\frac{H(a+k(b-a)\omega+(b-a)\omega)-H(a+k(b-a)\omega)}{(b-a)\omega}+\gamma_k\right)(b-a)\omega\right)\\
& & +H'(b)\cdot(b-a)\omega
\end{eqnarray*}
where $\gamma_k$ is 0 or infinitesimal for each $k$.  Writing $\beta$ for the $H'(b)(b-a)\omega$, which is 0 or infinitesimal, and simplifying yields the telescoping sum
\begin{eqnarray*}
& = & \beta+\sum_{k=1}^{\Omega-1}\left(H(a+(k+1)(b-a)\omega)-H(a+k(b-a)\omega)+\gamma_k(b-a)\omega\right)\\
& = & \beta + H(a+2(b-a)\omega)-H(a+1(b-a)\omega)+\gamma_1(b-a)\omega\\
&  & +H(a+3(b-a)\omega)-H(a+2(b-a)\omega)+\gamma_2(b-a)\omega\\
&  & +H(a+4(b-a)\omega)-H(a+3(b-a)\omega)+\gamma_3(b-a)\omega\\
& & +\dots\\
&  & +H(a+\Omega(b-a)\omega)-H(a+(\Omega-1)(b-a)\omega)+\gamma_{\Omega-1}(b-a)\omega\\
& = & \beta+H(a+\Omega(b-a)\omega)-H(a+1(b-a)\omega)+\sum_{k=1}^{\Omega-1}\gamma_k(b-a)\omega.
\end{eqnarray*}
By Lemma \ref{L1}, $\gamma=\sum_{k=1}^{\Omega-1}\gamma_k(b-a)\omega$ is 0 or infinitesimal.  Since $H$ is continuous at $a$ (it is differentiable there), $H(a+(b-a)\omega)=H(a)+\eta$ where $\eta$ is 0 or infinitesimal. 
We then arrive at
$$\sum_{k=1}^\Omega H'(a+k(b-a)\omega)\cdot (b-a)\omega=H(b)-H(a)-\eta+\beta+\gamma,$$
the standard part of which is $H(b)-H(a)$.
Since $\Omega$ was arbitrary the proof is complete.
\end{proof}





\end{document}